\documentclass{amsart}

\usepackage{amssymb}
\usepackage{mathtools}
\usepackage{hyperref}
\usepackage{comment}
\newtheorem{theorem}[equation]{Theorem}

\newtheorem{lemma}[equation]{Lemma}
\newtheorem{corollary}[equation]{Corollary}

\theoremstyle{definition}

\theoremstyle{plain}
\numberwithin{equation}{section}
\numberwithin{figure}{section}

\def\Q{\mathbb Q}
\def\Z{\mathbb Z}

\newcommand{\leg}[2]{\left(\frac{#1}{#2}\right)}
\newcommand{\stir}[2]{\genfrac{[}{]}{0pt}{}{#1}{#2}}

\begin{document}
\title[Legendre-symbol determinants and Bernoulli numbers]{Two Legendre-symbol determinants of Sun and Bernoulli numbers}
\author{Hang LIU}
\address[Hang Liu]{School of Mathematical Sciences, Shenzhen University,
Shenzhen, 518060, Guangdong, P. R. China}
\email{liuhang@szu.edu.cn}


\begin{abstract}
Let $p>5$ be an odd prime, put $n=(p-1)/2$, and let $\leg{\cdot}{p}$
be the Legendre symbol. Define $D_p^{\pm}$ to be the following determinants
\[
\begin{aligned}
D_p^{\pm}&:=\det\left[(i\pm j)\leg{i\pm j}{p}\right]_{1\leq i,j\leq n}.
\end{aligned}
\] 
In this article, we give explicit formulas for these determinants modulo $p$ in terms of Bernoulli numbers.
As a consequence of Reinhart's recent counterexample to the Ankeny--Artin--Chowla conjecture, we find a prime $p$ such that $D_p^{\pm} \equiv 0 \pmod p$. This gives a negative answer to a question of the Sun.
\end{abstract}

\maketitle

\section{Introduction}
Determinants with entries involving Legendre symbols are studied extensively and have deep connections with arithmetic invariants such as class numbers and fundamental units of quadratic fields (see, e.g., \cite{Chapman2004}, \cite{Sun2019}, \cite{GrinbergSunZhao2022}).

Let \(p>5\) be a prime and \(n=(p-1)/2\). Sun considered the following determinants in \cite[Conjecture~4.5(2)]{Sun2019}
$$
D_p^{\pm}:=\det\left[(i\pm j)\leg{i\pm j}{p}\right]_{1 \leq i, j \leq n}.
$$
He conjectured $D_p^{\pm} \not \equiv 0 \pmod p$, and $\leg{D_p^{\pm}}{p}=1$ for $p \equiv 1\pmod 4$.

In this paper, we will explicitly evaluate $D_p^{\pm}$ modulo $p$ in terms of Bernoulli numbers.

\begin{theorem}\label{thm:Dp}
Let \(p>5\) be a prime, and \(n=(p-1)/2\).  Then we have the following
congruences of $D_p^{+}$ and $D_p^{-}$ modulo \(p\):
\[
D_p^{+}\equiv
\begin{cases}
-\dfrac{9}{4}B_n^{\,2},
& p\equiv 1 \pmod 8,\\[6pt]
\phantom{-}\dfrac{9}{4}B_n^{\,2},
& p\equiv 5 \pmod 8,\\[6pt]
(-1)^{(h(-p)-1)/2}\dfrac{3}{8}B_{n-1}B_{n+1},
& p\equiv 3 \pmod 4,
\end{cases}
\pmod p,
\]
and
\[
D_p^{-}\equiv
\begin{cases}
-\dfrac{9}{4}B_n^{\,2},
& p\equiv 1 \pmod 8,\\[6pt]
(-1)^{(h(-p)-1)/2}
\dfrac{5}{32}B_{n-1}\bigl(24B_{n+1}-5B_{n-1}\bigr),
& p\equiv 3 \pmod 8,\\[6pt]
-\dfrac{81}{4}B_n^{\,2},
& p\equiv 5 \pmod 8,\\[6pt]
(-1)^{(h(-p)-1)/2}\dfrac{9}{32}B_{n-1}^{\,2},
& p\equiv 7 \pmod 8.
\end{cases}
\pmod p.
\]
Here \(B_m\) denotes the \(m\)-th Bernoulli number and \(h(-p)\) denotes the class number of
\(\Q(\sqrt{-p})\). All rational constants are interpreted modulo \(p\).
\end{theorem}

Let $p\equiv 1 \pmod 4$ and $K=\Q(\sqrt{p})$. Put \(\omega=(1+\sqrt p)/2\), and let
\(\varepsilon=x+y\omega>1\), with \(x,y\in\Z\), be the fundamental
unit of \(\mathcal O_K\).
The Ankeny--Artin--Chowla conjecture \cite{AnkenyArtinChowla1952} predicts that $p \nmid y$. However, Reinhart~\cite{Reinhart2024} recently found a counterexample to this conjecture, namely
$$p = 331914313984493.$$

Let $h(p)$ be the class number of $K$. The classical congruence $h(p)y\equiv (2x+y)B_{(p-1)/2} \pmod p$ shows $p \mid B_{(p-1)/2}$ for this prime.
Since $p\equiv 5 \pmod 8$, Theorem \ref{thm:Dp} shows $D_p^{\pm} \equiv 0 \pmod p$ which gives a counterexample to Sun's conjecture stated above.

The proof has two main ingredients. First, we reduce \(\det[(x_i+y_j)^{n+1}]_{1\leq i,j\leq n}\) to a \(2\times2\) determinant. Specializing $x_i=i$ and
$y_j=\pm j$, and applying Euler's criterion, then yields a formula
for $D_p^\pm\pmod p$. The entries in the \(2\times2\) determinant are sums involving unsigned Stirling numbers of the first kind.
Second, we express these entries as coefficients of power series involving powers of $F(t)=t/(e^t-1)$. These power series simplify modulo $p$, and their coefficients can then be evaluated in terms of Bernoulli numbers.

In Sections 2 and 3, we carry out the first and second steps, respectively. In Section 4, we complete the proof of Theorem \ref{thm:Dp} by combining the two ingredients.

\section{A formula for $\det[(x_i+y_j)^{n+1}]_{1\leq i,j\leq n}$ and a reduction of $D_p^{\pm} \pmod p$}
In this section, we reduce
$\det\bigl[(x_i+y_j)^{n+1}\bigr]_{1\le i,j\le n}$
to a $2\times2$ determinant and then specialize the resulting formula
to $D_p^\pm\pmod p$. We first use the binomial theorem to factor the
matrix as the product of an $n\times(n+2)$ matrix and an
$(n+2)\times n$ matrix. The Cauchy--Binet formula then expresses its
determinant as a sum involving two Vandermonde-type determinants. Using
the dual Jacobi--Trudi identity, we rewrite these determinants in terms
of elementary symmetric polynomials. A second application of the
Cauchy--Binet formula completes the reduction.

Now we need some notation. The $k$th elementary symmetric polynomial $e_k$ in \(z_1,\ldots,z_n\) is defined by 
$$
e_k(z_1,\ldots,z_n)
=
\sum_{1\leq i_1<\cdots<i_k\leq n}
\prod_{j=1}^{k}z_{i_j}.
$$
In addition, we use the conventions $e_0(z_1,\ldots,z_n)=1$ and $e_k(z_1,\ldots,z_n)=0$ for $k<0$ or $k>n$.
We also write
$$
\Delta(z_1,\ldots,z_n)
=
\prod_{1\leq i<j\leq n}(z_j-z_i).
$$
When there is no confusion, we abbreviate
\(e_k(x_1,\ldots,x_n)\), \(e_k(y_1,\ldots,y_n)\),
\(\Delta(x_1,\ldots,x_n)\), and \(\Delta(y_1,\ldots,y_n)\) to
\(e_k(x)\), \(e_k(y)\), \(\Delta(x)\), and \(\Delta(y)\), respectively.

\begin{theorem}
\label{thm:xy-det}
Let \(n\) be a positive integer, and let
\(x_1,\ldots,x_n,y_1,\ldots,y_n\) be indeterminates.  
Let \(C(x,y)=[C_{ij}(x,y)]_{1\leq i,j\leq2}\) be the $2 \times 2$ matrix with entries
$$
C_{ij}(x,y)
=
\sum_{r=0}^{n+1}
\frac{
e_{r+1-i}(y)\,
e_{n-r+j-1}(x)
}{
\binom{n+1}{r}
}.
$$
Then we have
\begin{equation}
\label{eq:xy-det}
\begin{aligned}
\det\bigl[(x_i+y_j)^{n+1}\bigr]_{1\leq i,j\leq n}
&=
(-1)^{n(n-1)/2}
\Delta(x)\Delta(y) \\
&\qquad\cdot
\left(\prod_{r=0}^{n+1}\binom{n+1}{r}\right)
\det C(x,y).
\end{aligned}
\end{equation}
\end{theorem}

\begin{proof}
Define an $n \times (n+2)$ matrix $U$ and an $(n+2) \times n$ matrix $V$ by
$$U=\left[\binom{n+1}{r}x_i^r\right]_{\substack{1\leq i\leq n \\ 0\leq r\leq n+1}} \text{  and  } V=\left[y_j^{n+1-r}\right]_{\substack{0\leq r\leq n+1 \\ 1\leq j\leq n}}.$$

The binomial theorem yields
$$
(x_i+y_j)^{n+1}
=
\sum_{r=0}^{n+1}\binom{n+1}{r}x_i^r y_j^{n+1-r},
$$
which gives
$$\bigl[(x_i+y_j)^{n+1}\bigr]_{1\leq i,j\leq n} = UV.$$

By the Cauchy--Binet formula, we therefore get
\begin{equation}\label{eq:detsimple}
\det[(x_i+y_j)^{n+1}]
=
\sum_{0\leq a<b\leq n+1}
\left(\prod_{\substack{0\leq r\leq n+1\\ r\neq a,b}}\binom{n+1}{r}\right)
X_{a,b}(x)Y_{a,b}(y),  
\end{equation}
where $X_{a,b}(x)$ and $Y_{a,b}(y)$ are the following $n \times n$ determinants
\begin{align*}
X_{a,b}(x)
&=
\det\bigl[x_i^r\bigr]_{\substack{1\leq i\leq n\\
r\in\{0,\ldots,n+1\}\setminus\{a,b\}}}, \\
Y_{a,b}(y)
&=
\det\bigl[y_j^{n+1-r}\bigr]_{\substack{
r\in\{0,\ldots,n+1\}\setminus\{a,b\}\\
1\leq j\leq n}}.
\end{align*}

We now evaluate $X_{a,b}(x)$.  Let \(r_1<r_2<\cdots<r_n\) be the
elements of \(\{0,1,\ldots,n+1\}\setminus\{a,b\}\).
Then we have 
\begin{equation}\label{eq:Xab}
  X_{a,b}(x) = \Delta(x_1,\ldots,x_n)s_\lambda(x_1,\ldots,x_n),
\end{equation}
where $s_\lambda$ is the Schur polynomial and $\lambda=(2^{\,n+1-b},1^{\,b-a-1})$ is a partition with exponents denoting multiplicities.
By the dual Jacobi--Trudi identity~\cite[Chapter~I, Section~3]{Macdonald1995}, we have
\begin{equation}\label{eq:sx}
s_\lambda(x_1,\ldots,x_n)
=
\det\bigl[e_{\lambda'_i-i+j}(x_1,\ldots,x_n)\bigr]_{1\leq i,j\leq2},
\end{equation}
where $\lambda'=(n-a,\ n+1-b)$ is the dual partition of $\lambda$.

Hence, combining \eqref{eq:Xab} and \eqref{eq:sx} we have
$$
X_{a,b}(x)
=
\Delta(x_1,\ldots,x_n)
\begin{vmatrix}
e_{n-a}(x) & e_{n-a+1}(x)\\
e_{n-b}(x) & e_{n-b+1}(x)
\end{vmatrix}.
$$

Similarly, for \(Y_{a,b}(y)\), the exponents are \(n+1-r\).  Reversing their
order contributes the sign \((-1)^{n(n-1)/2}\).  After the reversal, the
deleted exponents are \(n+1-b<n+1-a\). Thus the corresponding Schur polynomial has partition $\mu=(2^a,1^{\,b-a-1})$ with dual $\mu'=(b-1,\ a)$.
Therefore we have
$$
Y_{a,b}(y)
=
(-1)^{n(n-1)/2}
\Delta(y_1,\ldots,y_n)
\begin{vmatrix}
e_a(y) & e_b(y)\\
e_{a-1}(y) & e_{b-1}(y)
\end{vmatrix}.
$$

Substituting these evaluations into \eqref{eq:detsimple} gives
\begin{equation}
\label{eq:xy-det-after-minors}
\begin{aligned}
\det[(x_i+y_j)^{n+1}]
&=
(-1)^{n(n-1)/2}\Delta(x)\Delta(y)
\left(\prod_{r=0}^{n+1}\binom{n+1}{r}\right)\\
& \times
\sum_{0\leq a<b\leq n+1}
\frac{1}{\binom{n+1}{a}\binom{n+1}{b}}
\begin{vmatrix}
e_{n-a}(x) & e_{n-a+1}(x)\\
e_{n-b}(x) & e_{n-b+1}(x)
\end{vmatrix}\cdot
\begin{vmatrix}
e_a(y) & e_b(y)\\
e_{a-1}(y) & e_{b-1}(y)
\end{vmatrix}.
\end{aligned}
\end{equation}

Let \(A\) be the \(2\times(n+2)\) matrix and \(B\) the
\((n+2)\times2\) matrix defined by
\begin{equation*}
A=
\left[
\frac{e_{r+1-i}(y)}{\binom{n+1}{r}}
\right]_{\substack{1\le i\le 2\\ 0\le r\le n+1}} \text{   and   }
B=
\left[
e_{n-r+j-1}(x)
\right]_{\substack{0\le r\le n+1\\ 1\le j\le 2}}.
\end{equation*}
Applying the Cauchy--Binet formula again, we have
\begin{align*}
&\sum_{0\le a<b\le n+1}
\frac{1}{\binom{n+1}{a}\binom{n+1}{b}}
\begin{vmatrix}
e_{n-a}(x) & e_{n-a+1}(x)\\
e_{n-b}(x) & e_{n-b+1}(x)
\end{vmatrix}
\begin{vmatrix}
e_a(y) & e_b(y)\\
e_{a-1}(y) & e_{b-1}(y)
\end{vmatrix} \\
&\quad =
\sum_{0\le a<b\le n+1}
\det A_{\{a,b\}}\det B_{\{a,b\}} \\
&\quad =
\det(AB)\\
&\quad =
\det
\begin{pmatrix}
\displaystyle\sum_{r=0}^{n+1}
\frac{e_r(y)e_{n-r}(x)}{\binom{n+1}{r}}
&
\displaystyle\sum_{r=0}^{n+1}
\frac{e_r(y)e_{n-r+1}(x)}{\binom{n+1}{r}}
\\[4mm]
\displaystyle\sum_{r=0}^{n+1}
\frac{e_{r-1}(y)e_{n-r}(x)}{\binom{n+1}{r}}
&
\displaystyle\sum_{r=0}^{n+1}
\frac{e_{r-1}(y)e_{n-r+1}(x)}{\binom{n+1}{r}}
\end{pmatrix} \\
&\quad =
\det C(x,y).
\end{align*}
The theorem now follows.
\end{proof}

Next we want to give a formula of $D_p^{\pm} \pmod p$ by specializing $x_i=i$ and $y_j=\pm j$ in Theorem~\ref{thm:xy-det}. But before this, we give a formula for the scaling factor in the Theorem after the specialization.
\begin{lemma}
\label{lem:scalar-factor}
Let \(p>3\) be a prime, $n=(p-1)/2$ and let \(h(-p)\) denote the class number of \(\Q(\sqrt{-p})\).
Define
$$
P_n=
\left(\prod_{j=1}^{n-1}j!\right)^2
\left(\prod_{r=0}^{n+1}\binom{n+1}{r}\right),
$$
and
$$
\varepsilon_p
=
\begin{cases}
-1, & p\equiv1\pmod4,\\[4pt]
\displaystyle
\leg{2}{p}(-1)^{(h(-p)+1)/2}, & p\equiv3\pmod4,
\end{cases}
$$
where \(\leg{2}{p}\) is the Legendre symbol.  Then we have
$$
P_n\equiv\varepsilon_p\pmod p.
$$
\end{lemma}
\begin{proof}
We have 
\begin{equation}\label{eq:pn}
P_n = \left(\prod_{j=1}^{n-1}j!\right)^2 \frac{((n+1)!)^{n+2}}
{\left(\prod_{r=0}^{n+1}r!\right)^2} 
=
\frac{((n+1)!)^{n+2}}{(n!)^2(n+1)!^2}
=
(n+1)^n(n!)^{n-2}.
\end{equation}
Since $n+1=(p+1)/2\equiv 2^{-1} \pmod p$, Euler's criterion gives
\begin{equation}\label{eq:nn}
(n+1)^n
\equiv
2^{-n} 
\equiv
\leg{2}{p}
\pmod p.
\end{equation}
On the other hand, Wilson's theorem gives
\begin{equation}\label{eq:nf}
(n!)^2\equiv(-1)^n(p-1)!\equiv(-1)^{n+1}\pmod p.
\end{equation}

If $p\equiv 1 \pmod 4$, then $n$ is even. 
Substituting \eqref{eq:nn} and \eqref{eq:nf} into \eqref{eq:pn}, we have
$$
P_n
\equiv
\leg{2}{p}(-1)^{(n-2)/2}
\pmod p.
$$
It is easy to check \(P_n\equiv -1\pmod p\) in this case.

If \(p\equiv3\pmod4\), then \(n\) is odd.
Similarly, we have
$$
P_n
\equiv
\leg{2}{p}n! = \leg{2}{p}\left(\frac{p-1}{2}\right)!
\pmod p.
$$
By the classical congruence of Mordell~\cite{Mordell1961}, we have
$\left(\frac{p-1}{2}\right)!
\equiv
(-1)^{(h(-p)+1)/2}
\pmod p$ for \(p\equiv3\pmod4\),
which gives
$$
P_n
\equiv
\leg{2}{p}(-1)^{(h(-p)+1)/2}
\pmod p.
$$
The two cases prove the lemma.
\end{proof}

Now we introduce some notation.
Let \(\stir{m}{k}\) be the unsigned Stirling numbers of the first kind, defined by
$$
t(t+1)\cdots(t+m-1)
=
\sum_{k=0}^m \stir{m}{k}t^k.
$$
We use the convention
$\stir{m}{k}=0$
for $k<0$ or $k>m$.
By the definition of unsigned Stirling numbers and Vieta's formulas, it is easy to see \(\stir{m}{k}=e_{m-k} (1,2,\ldots,m-1)\).

\begin{corollary}
\label{cor:sun-det-stirling}
Let \(p>3\) be a prime, and \(n=(p-1)/2\).  Let
\(\varepsilon_p\) be as in Lemma~\ref{lem:scalar-factor}.  Define two
\(2\times2\) matrices \(A_p^+\) and \(A_p^-\) by
$$
(A_p^+)_{ij}
=
\sum_{r=0}^{n+1}
\frac{
\stir{n+1}{n-r+i}\stir{n+1}{r-j+2}
}{
\binom{n+1}{r}
}
\qquad(1\leq i,j\leq2),
$$
and
$$
(A_p^-)_{ij}
=
\sum_{r=0}^{n+1}
\frac{
(-1)^{r+1-i}\stir{n+1}{n-r+i}\stir{n+1}{r-j+2}
}{
\binom{n+1}{r}
}
\qquad(1\leq i,j\leq2).
$$
Then we have
$$
D_p^+
\equiv
(-1)^{n(n-1)/2}\varepsilon_p\det A_p^+
\pmod p,
$$
and
$$
D_p^-
\equiv
\varepsilon_p\det A_p^-
\pmod p,
$$
where the fractions are interpreted modulo \(p\).
\end{corollary}
\begin{proof}
By Euler's criterion, we have
\begin{equation*}
D_p^{\pm}
\equiv
\det[(i\pm j)^{n+1}]_{1\leq i,j\leq n}
\pmod p.
\end{equation*}

For \(D_p^+\), we apply Theorem~\ref{thm:xy-det} with
\(x_i=i\) and \(y_j=j\) for \(1\leq i,j \leq n\).  
Since $\Delta(1,2,\ldots,n) = \prod_{j=1}^{n-1}j!$, substituting the symmetric polynomials of $1,2,\ldots,n$ by unsigned Stirling numbers, we get 
$$
D_p^+
\equiv
(-1)^{n(n-1)/2}P_n\det A_p^+ 
\pmod p.
$$
The result for \(D_p^+\) now follows from Lemma~\ref{lem:scalar-factor}.

Similarly, for \(D_p^-\), we apply Theorem~\ref{thm:xy-det} with
\(x_i=i\) and \(y_j=-j\) for \(1\leq i,j\leq n\). In this case, we have
$$
\Delta(-1,-2,\ldots,-n)
=
(-1)^{n(n-1)/2}\Delta(1,2,\ldots,n),
$$
and
$$
e_{r+1-i}(-1,-2,\ldots,-n)
=
(-1)^{r+1-i}e_{r+1-i}(1,2,\ldots,n).
$$
Substituting these into Theorem~\ref{thm:xy-det} and using Lemma~\ref{lem:scalar-factor} again, we get 
$$
D_p^-
\equiv
P_n\det A_p^- \equiv \varepsilon_p\det A_p^-
\pmod p.
$$
\end{proof}

\section{Congruences for \(A_p^{\pm}\) using Bernoulli numbers}
The next
lemma expresses the matrices \(A_p^{\pm}\) in Corollary~\ref{cor:sun-det-stirling} modulo \(p\) in terms of Bernoulli numbers and
values of Bernoulli polynomials at \(1/2\).

\begin{lemma}
\label{lem:entries-Apm}
Let \(p>3\) be a prime, \(n=(p-1)/2\), and \(A_p^{\pm}\) be the matrices defined in Corollary~\ref{cor:sun-det-stirling}.
Let \(B_k(x)\) denote the \(k\)-th Bernoulli polynomial, and 
\(B_k=B_k(0)\) be the \(k\)-th Bernoulli number.  Then we have
$$
A_p^+
\equiv
\begin{pmatrix}
\dfrac32 B_n
&
\dfrac38\bigl(B_{n+1}-B_n\bigr)
\\[6pt]
(-1)^nB_{n-1}
&
\dfrac32 B_n
\end{pmatrix} \pmod p,
$$
and
$$
A_p^-
\equiv
\begin{pmatrix}
\dfrac32 B_n\left(\dfrac12\right)
-\dfrac14 B_{n-1}\left(\dfrac12\right)
&
\dfrac1{32}B_{n-1}\left(\dfrac12\right)
+\dfrac38 B_{n+1}\left(\dfrac12\right)
\\[8pt]
-B_{n-1}\left(\dfrac12\right)
&
\dfrac32 B_n\left(\dfrac12\right)
+\dfrac14 B_{n-1}\left(\dfrac12\right)
\end{pmatrix} \pmod p,
$$
where the rational numbers are interpreted modulo \(p\).
\end{lemma}
\begin{proof}
Let $F(t)=\frac{t}{e^t-1}$, and $\theta=t\frac{d}{dt}$ be the Euler operator. By the definition of Bernoulli numbers, we have 
\begin{equation}\label{eq:Fseries}
  F(t) = \sum_{k\ge0}B_k\frac{t^k}{k!}.
\end{equation}
We also have the following obvious identities
\begin{align}
F(t)&=F(-t)-t \label{eq:ffm} \\
\theta F(t)&=F(t)(1-F(-t)), \label{eq:thetaf}\\
\theta F(-t)&=F(-t)(1-F(t)) \label{eq:thetafm}.
\end{align}

For a formal power series \(G(t)=\sum_{j\ge0}c_jt^j\), let $[t^r]G(t)=c_r$ be the coefficient of $t^r$.

We divide the proof into three steps. In the first, we rewrite the
entries of \(A_p^\pm\) as coefficients of power series involving
\(F(t)\) and \(F(-t)\). In the second, we reduce the high powers
\(F(t)^{2n+2}\) and \((F(t)F(-t))^{n+1}\) modulo \(p\).
In the third step, we express the resulting power series as linear combinations of power series whose coefficients are Bernoulli numbers or values of Bernoulli polynomials at $1/2$, and then extract the required coefficients.

\noindent\underline{Step 1.}
The higher-order Bernoulli numbers $B_r^{(n+1)}$ are defined by
\[
F(t)^{n+1}
=
\sum_{r\ge0}B_r^{(n+1)}\frac{t^r}{r!}.
\]
We have the well-known relation between higher order Bernoulli numbers and Stirling numbers
\begin{equation}
\label{eq:stirling-B-higher}
\stir{n+1}{n-r+1}
=
(-1)^r\binom nr B_r^{(n+1)}
\qquad(0\le r\le n).
\end{equation}

We first consider \((A_p^+)_{11}\). By the definition of $A_p^+$ and \eqref{eq:stirling-B-higher}, we have
\begin{align*}
(A_p^+)_{11}
&=
(-1)^n
\sum_{r=0}^{n}
\frac{\binom nr\binom n{n-r}}{\binom{n+1}{r}}
B_r^{(n+1)}B_{n-r}^{(n+1)} \\
&=
(-1)^n\frac{n!}{n+1}
\sum_{r=0}^{n}
\left\{
\frac{B_r^{(n+1)}}{r!}
\frac{B_{n-r}^{(n+1)}}{(n-r)!}
+
\frac{B_r^{(n+1)}}{r!}
\frac{(n-r)B_{n-r}^{(n+1)}}{(n-r)!}
\right\} \\
&=
(-1)^n\frac{n!}{n+1}[t^n]\,
F(t)^{n+1}(1+\theta)\bigl(F(t)^{n+1}\bigr) \\
&= (-1)^n n![t^n]
\left(1+\frac1{n+1}-F(-t)\right)F(t)^{2n+2},
\end{align*}
where the second equality follows from
\[
\frac{\binom nr\binom n{n-r}}{\binom{n+1}{r}}
=
\frac{1+(n-r)}{n+1}\binom nr,
\]
and the last follows from \eqref{eq:thetaf}.

The same method gives the following formulas:
\begin{equation}
\label{eq:Aplus-coeff-forms}
\begin{aligned}
(A_p^+)_{11}
&=
(-1)^n n![t^n]
\left(1+\frac1{n+1}-F(-t)\right)F(t)^{2n+2},\\
(A_p^+)_{12}
&=
(-1)^{n+1}(n+1)![t^{n+1}]
(1-F(-t))^2F(t)^{2n+2},\\
(A_p^+)_{21}
&=
(-1)^{n-1}\frac n{n+1}(n-1)![t^{n-1}]F(t)^{2n+2},\\
(A_p^+)_{22}
&=
(-1)^n n![t^n]
\left(1+\frac1{n+1}-F(-t)\right)F(t)^{2n+2}.
\end{aligned}
\end{equation}

Next we consider \((A_p^-)_{11}\). 
By \eqref{eq:stirling-B-higher}, we have
\begin{align*}
(A_p^-)_{11}
&=
\sum_{r=0}^{n}
(-1)^r
\frac{r+1}{n+1}\binom nr
B_{n-r}^{(n+1)}B_r^{(n+1)} \\
&= \frac{n!}{n+1}
\sum_{r=0}^{n}
\left\{
\frac{B_{n-r}^{(n+1)}}{(n-r)!}
\frac{(-1)^rB_r^{(n+1)}}{r!}
+
\frac{B_{n-r}^{(n+1)}}{(n-r)!}
\frac{r(-1)^rB_r^{(n+1)}}{r!}
\right\} \\
&=\frac{n!}{n+1}[t^n]\,
F(t)^{n+1}(1+\theta)\bigl(F(-t)^{n+1}\bigr) \\
&=n![t^n]\left(1+\frac1{n+1}-F(t)\right)
(F(t)F(-t))^{n+1}
\end{align*}
where we use \eqref{eq:thetafm} for the last equality.

Similarly, we have
\begin{equation}
\label{eq:Aminus-coeff-forms}
\begin{aligned}
(A_p^-)_{11}
&=
n![t^n]\left(1+\frac1{n+1}-F(t)\right)
(F(t)F(-t))^{n+1},\\
(A_p^-)_{12}
&=
(n+1)![t^{n+1}]
(1-F(t))(1-F(-t))(F(t)F(-t))^{n+1},\\
(A_p^-)_{21}
&=
\frac n{n+1}(n-1)![t^{n-1}]
(F(t)F(-t))^{n+1},\\
(A_p^-)_{22}
&=
n![t^n]\left(1+\frac1{n+1}-F(-t)\right)
(F(t)F(-t))^{n+1}.
\end{aligned}
\end{equation}

\noindent\underline{Step 2.}
Since $n+1=\frac{p+1}{2}<p-1$ for $p>3$, the von Staudt--Clausen theorem shows
that the coefficients of \(F(t)\) up to degree \(n+1\) are
\(p\)-integral. We work with the truncation of \(F(t)\) modulo
\(t^{n+2}\). 
Writing $F(t)\equiv 1+tG(t)$ modulo
\(t^{n+2}\) where $G(t)$ has $p$-integral coefficients, we have
\begin{equation}\label{eq:mod-F-power}
\begin{aligned}
F(t)^{2n+2}
&=F(t)^{p+1}
 \equiv F(t)(1+tG(t))^p\\
&\equiv F(t)\bigl(1+t^pG(t)^p\bigr)
 \equiv F(t)\pmod{(p,t^{n+2})},
\end{aligned}
\end{equation}
since $p\ge n+2$.

Let $H(t)=e^{t/2}F(t)$. By the definition of Bernoulli polynomials, we have
\begin{equation}\label{eq:Hseries}
H(t) = \sum_{k\ge0}B_k\!\left(\frac12\right)\frac{t^k}{k!},
\end{equation}
where $B_k(x)$ is the $k$th Bernoulli polynomial.

Since $(n+1) \equiv \frac12 \pmod p$ and $k!$ is invertible modulo $p$ for
$k\le n+1$, we have
\[
e^{(n+1)t}\equiv e^{t/2}\pmod{(p,t^{n+2})}.
\]
Using \(F(-t)=e^tF(t)\), we obtain
\begin{equation}
\label{eq:mod-FFminus-power}
(F(t)F(-t))^{n+1}
=
e^{(n+1)t}F(t)^{2n+2}
\equiv H(t) \pmod{(p,t^{n+2})}.
\end{equation}

\noindent\underline{Step 3.}
After substituting \eqref{eq:mod-F-power} into
\eqref{eq:Aplus-coeff-forms}, we only need to extract the coefficients
of \(F(t)\), \(F(t)F(-t)\), and \((1-F(-t))^2F(t)\).
The coefficients of $F(t)$ are given in \eqref{eq:Fseries}. Now we consider the other two power series.

By \eqref{eq:thetaf}, we have $F(t)F(-t)=F(t)-\theta F(t)$
which gives the following identity by taking coefficients
\begin{equation}
\label{eq:coeff-ordinary}
n![t^n]F(t)F(-t)=(1-n)B_n.
\end{equation}

Applying $\theta$ once more to \eqref{eq:thetaf} and rearranging using \eqref{eq:ffm}, 
we have
\[
(1-F(-t))^2F(t)
=
\frac12\{\theta^2F(t)+\theta F(t)-t\theta F(t)+tF(t)\}.
\]
Taking the coefficient of \(t^{n+1}\), multiplying by \((-1)^{n+1}\), and using \(B_m=0\) for odd \(m>1\),
we get
\begin{equation}
\label{eq:coeff-Aplus12}
\begin{aligned}
&(-1)^{n+1}(n+1)![t^{n+1}]
(1-F(-t))^2F(t)\\
&\qquad =
\frac{n+1}{2}\left((n+2)B_{n+1}+(n-1)B_n\right).
\end{aligned}
\end{equation}

Similarly, after substituting \eqref{eq:mod-FFminus-power} into
\eqref{eq:Aminus-coeff-forms}, we need to extract the coefficients
of \(H(t)\), \(F(-t)H(t)\), \(F(t)H(t)\), and
\((1-F(t))(1-F(-t))H(t)\).
The coefficients of $H(t)$ are given in \eqref{eq:Hseries}. We now express the other three power series using $H(t), \theta H(t)$ and $\theta^2 H(t)$.

Applying $\theta$ to $H(t)$, we have 
\begin{equation}
\label{eq:thetaH}
\frac{\theta H(t)}{H(t)}
=
\frac t2+1-F(-t),
\end{equation}
which gives 
\begin{equation}\label{eq:FminusH}
F(-t)H(t)=\left(1+\frac t2\right)H(t)-\theta H(t).
\end{equation}
Taking the coefficient of $t^n$, we get
\begin{equation}
\label{eq:coeff-FminusH}
n![t^n]F(-t)H(t)
=
(1-n)B_n\!\left(\frac12\right)
+
\frac n2B_{n-1}\!\left(\frac12\right).
\end{equation}

Substituting \eqref{eq:ffm} into \eqref{eq:FminusH} and taking the coefficient, we have
\begin{equation}
\label{eq:coeff-FH}
n![t^n]F(t)H(t)
=
(1-n)B_n\!\left(\frac12\right)
-
\frac n2B_{n-1}\!\left(\frac12\right).
\end{equation}

It remains to consider $(1-F(t))(1-F(-t))H(t)$.
By \eqref{eq:thetaH} and \eqref{eq:ffm}, we have 
\[
1-F(t)=\frac{\theta H(t)}{H(t)}+\frac t2,
\qquad
1-F(-t)=\frac{\theta H(t)}{H(t)}-\frac t2.
\]
Multiplying these identities gives
\[
(1-F(t))(1-F(-t))H(t)
=
\frac{(\theta H(t))^2}{H(t)}-\frac{t^2}{4}H(t).
\]

Applying $\theta$ to \eqref{eq:FminusH} and rearranging using \eqref{eq:thetafm}, we have 
\[
\frac{(\theta H(t))^2}{H(t)}
=
\frac12\theta^2H(t)
+\frac12\theta H(t)
+\frac{t^2}{8}H(t).
\]
Substituting this into the previous formula yields
\[
(1-F(t))(1-F(-t))H(t)
=
\frac12\theta^2H(t)
+
\frac12\theta H(t)
-
\frac{t^2}{8}H(t).
\]
Taking the coefficient of \(t^{n+1}\), we finally get
\begin{equation}
\label{eq:coeff-crossH}
\begin{aligned}
&(n+1)![t^{n+1}]
(1-F(t))(1-F(-t))H(t)\\
&\qquad =
\frac{(n+1)(n+2)}2B_{n+1}\!\left(\frac12\right)
-
\frac{n(n+1)}8B_{n-1}\!\left(\frac12\right).
\end{aligned}
\end{equation}

Substituting the reductions from Step 2 and preceding coefficient identities into
\eqref{eq:Aplus-coeff-forms} and \eqref{eq:Aminus-coeff-forms},
and using $n\equiv-1/2 \pmod p$ gives the stated congruences of $A_p^{\pm}$.
\end{proof}

\section{Proof of Theorem \ref{thm:Dp}}
We combine Corollary~\ref{cor:sun-det-stirling} and
Lemma~\ref{lem:entries-Apm} to prove Theorem \ref{thm:Dp}.
\begin{proof}[Proof of Theorem \ref{thm:Dp}]
Recall that by Corollary~\ref{cor:sun-det-stirling} we have
$$
D_p^+
\equiv
(-1)^{n(n-1)/2}\varepsilon_p\det A_p^+
\pmod p,
\qquad 
D_p^-
\equiv
\varepsilon_p\det A_p^-
\pmod p,
$$
where $\varepsilon_p$ is given by Lemma \ref{lem:scalar-factor}.

By Lemma~\ref{lem:entries-Apm} and the vanishing of \(B_m\) for odd
\(m>1\), we have
\[
\det A_p^+
\equiv
\begin{cases}
\dfrac94 B_n^2, & n\ \text{even},\\[5pt]
\dfrac38 B_{n-1}B_{n+1}, & n\ \text{odd}.
\end{cases}
\]
The value of \((-1)^{n(n-1)/2}\) is determined by \(p\bmod 8\).
This gives the stated formula for \(D_p^+\).

For \(D_p^-\), put \(C_m=B_m(1/2)\). A direct calculation from
Lemma~\ref{lem:entries-Apm} gives
\[
\det A_p^-
\equiv
\frac94 C_n^2
+\frac38 C_{n-1}C_{n+1}
-\frac1{32}C_{n-1}^2.
\]
We also have the standard identity
$C_m=(2^{1-m}-1)B_m$, while Euler's criterion gives
$2^n\equiv\left(\frac2p\right)\pmod p$.
Since \(\left(\frac2p\right)\) is determined by \(p\bmod 8\) and
\(B_m=0\) for odd \(m>1\), considering the four residue classes of
\(p\) modulo \(8\) gives the stated formula for \(D_p^-\).
\end{proof}

\subsection*{Acknowledgments}
I thank Professor Sun Zhi-Wei for very helpful discussions and comments.

The author was supported by the National Natural
Science Foundation of China under Grant No.~12371031.

\subsection*{AI use declaration} During preliminary exploration of this work, ChatGPT-5.5 Thinking suggested reducing the $n \times n$ determinant to a $2 \times 2$ determinant and exploring its connection to Bernoulli numbers. The author independently developed all results, verified all proofs, wrote the manuscript, and take full responsibility for its contents.


\begin{thebibliography}{00}
\bibitem{AnkenyArtinChowla1952}
N. C. Ankeny, E. Artin and S. Chowla,
\emph{The class-number of real quadratic number fields},
Ann. of Math. (2) \textbf{56} (1952), 479--493.

\bibitem{Chapman2004}
R. Chapman,
\emph{Determinants of Legendre symbol matrices},
Acta Arith. \textbf{115} (2004), no. 3, 231--244.

\bibitem{GrinbergSunZhao2022}
D. Grinberg, Z.-W. Sun and L. Zhao,
\emph{Proof of three conjectures on determinants related to quadratic residues},
Linear Multilinear Algebra \textbf{70} (2022), no. 19, 3734--3746.


\bibitem{Macdonald1995}
I. G. Macdonald,
\emph{Symmetric Functions and Hall Polynomials},
2nd ed.,
Oxford University Press, Oxford, 1995.

\bibitem{Mordell1961}
L. J. Mordell,
\emph{The congruence \(((p-1)/2)!\equiv \pm 1 \pmod p\)},
Amer. Math. Monthly \textbf{68} (1961), 145--146.

\bibitem{Reinhart2024}
A. Reinhart,
\emph{A counterexample to the Conjecture of Ankeny, Artin and Chowla},
arXiv:2410.21864v3, 2025.

\bibitem{Sun2019}
Z.-W. Sun,
\emph{On some determinants with Legendre symbol entries},
Finite Fields Appl. \textbf{56} (2019), 285--307.

\end{thebibliography}
\end{document}